\documentclass[10pt]{article}
\usepackage[a4paper,textwidth=120mm,textheight=195mm]{geometry}
\usepackage{amsmath,amssymb,amsthm,mathtools,bm}
\usepackage{booktabs,graphicx,subcaption,array}
\usepackage{algorithm,algpseudocode}
\usepackage{enumitem,microtype,xcolor}
\usepackage{hyperref,cleveref}
\usepackage[backend=biber,style=numeric,sorting=none,giveninits=true,doi=false,url=false,isbn=false,maxbibnames=99]{biblatex}
\AtBeginBibliography{\fontsize{8}{8.2}\selectfont\setlength{\bibitemsep}{0pt}\setlength{\bibparsep}{0pt}}
\setlist{nosep,leftmargin=1.6em}
\hypersetup{colorlinks=true,linkcolor=blue!55!black,citecolor=blue!55!black,urlcolor=blue!55!black,
  pdftitle={Certified Seventh-Order Two-Derivative Hermite Deferred Correction via Node-Sweep Co-Design},
  pdfauthor={Zhixin Huo}}

\newcommand{\dt}{\Delta t}
\newcommand{\Rone}{\mathcal R^{(1)}}
\newcommand{\Rtwo}{\mathcal R^{(2)}}
\newcommand{\OO}{\mathcal O}

\newcommand{\RR}{\mathbb R}
\newcommand{\norm}[1]{\left\lVert#1\right\rVert}
\newcommand{\Cseven}{C_7}
\newcommand{\Jstiff}{J_{\mathrm{stiff}}}
\newcommand{\Jtree}{J_{\mathrm{tree},2}}
\newcommand{\Minf}{M_{\infty}}

\theoremstyle{plain}
\newtheorem{theorem}{Theorem}[section]
\newtheorem{lemma}[theorem]{Lemma}
\newtheorem{proposition}[theorem]{Proposition}
\newtheorem{corollary}[theorem]{Corollary}
\theoremstyle{definition}

\newtheorem{remark}[theorem]{Remark}

\title{Certified Seventh-Order Two-Derivative\\
Hermite Deferred Correction via Node--Sweep Co-Design}
\author{Zhixin Huo\thanks{Corresponding author: \texttt{zhixinhuo@hpu.edu.cn}}\\
School of Mathematics and Information Science,\\
Henan Polytechnic University, Jiaozuo 454000, Henan, China}
\date{}

\begin{document}
\maketitle

\begin{abstract}
Two-derivative Hermite deferred correction combines high collocation order with sequential single-state solves, but the stopped method depends jointly on the nodes and the correction sweep.  We co-design these ingredients for diffusion-dominated semilinear problems.  For three subintervals, an H4 predictor, and two corrections, the complete order-seven B-series defect has rank one in the 48-dimensional rooted-tree space: two corrections apply two unary graftings to the one-directional order-five predictor defect.  Hence one scalar chain coefficient controls every nonlinear principal-error condition.  Rational nodes and an isolated algebraic correction parameter cancel this coefficient; coprimality with the order-eight chain polynomial proves classical order exactly seven.  A complementary design, Accuracy-P40, retains generic sixth order but reduces the complete principal-error norm to $9.8\%$ of the LGL--L3 value while satisfying $J_{\mathrm{stiff}}<0.40$.  We also distinguish convergence of repeated corrections from absolute stability after a fixed number of sweeps: two corrections have finite negative-real stability intervals, whereas a third correction restores far-stiff output damping for the new designs.  High-precision nonlinear order tests verify sixth versus seventh order, and Allen--Cahn and Cahn--Hilliard calculations show that Accuracy-P40 reduces correction and Krylov work.
\end{abstract}


\section{Introduction}
\label{sec:intro}
Consider the autonomous initial-value problem
\begin{equation}
 u_t=\Rone(u),\qquad
 \Rtwo(u)=\frac{d}{dt}\Rone(u(t))=\Rone_u(u)\Rone(u),
 \label{eq:ode}
\end{equation}
so that $u_{tt}=\Rtwo(u)$.  Two-derivative data arise naturally in Lax--Wendroff, generalized Riemann, ADER, gas-kinetic, and multiderivative discontinuous Galerkin discretizations \cite{ChanTsai2010,LiDu2016,ChristliebEtAl2016,GrantGottliebSeal2019,JaustSchuetzSeal2016,Dittmann2021}.  In such settings the same local space--time evolution that supplies a flux or residual can also supply its time derivative, so higher temporal order need not be purchased only through additional stages.

Deferred correction takes a complementary route.  A dense collocation equation is approached through simpler preconditioned updates, permitting a tradeoff among formal order, serial work, parallelism, and robustness \cite{Dutt2000,Minion2003,ChristliebOngQiu2010,ChristliebMacdonaldOng2010,SpeckEtAl2015,Weiser2015,CausleySeal2019,CaklovicEtAl2025,BoltenWimmer2026}.  Arbitrary-order two-derivative Hermite predictor--corrector and deferred-correction methods are established prior art \cite{SchuetzSealZeifang2022,ZeifangSchuetz2022,ZeifangSchuetzSeal2022}.  The question addressed here is narrower and more structural: when only two corrections are affordable, can the temporal nodes and the sequential two-derivative sweep be co-designed so that the full nonlinear principal defect collapses to low dimension, and can that collapse be converted into a rigorously certified order gain?

The target application class is the method-of-lines system
\begin{equation}
 u_t=A_hu+N_h(u),
 \label{eq:semilinear}
\end{equation}
where $A_h$ has a broad dissipative spectrum.  A second-order diffusion discretization has grid-scale eigenvalues of magnitude $\OO(h_x^{-2})$, while Cahn--Hilliard-type fourth-order dynamics can produce $\OO(h_x^{-4})$ stiffness.  One macrostep therefore contains slowly evolving physical modes and strongly damped grid-scale modes.  The dense Hermite quadrature is determined by the nodes, whereas the stopped iteration is governed by the interaction of those nodes with the local preconditioner.  Optimizing either ingredient in isolation does not optimize the complete method.

\paragraph{Central structural contribution.}
For three subintervals, the fourth-order Hermite predictor, and exactly two corrections, a generic sixth-order method has 48 possible rooted-tree defects at local order seven.  The central theorem proves that all 48 defects lie in a single direction: the H4 predictor has a one-directional order-five defect, and the leading action of each correction is a tree-independent scalar stage recurrence followed by unary grafting.  Two corrections therefore map the predictor defect into one one-dimensional order-seven subspace.  Consequently, the scalar chain coefficient $\Cseven$ is not merely a Dahlquist diagnostic; it controls the complete nonlinear principal error.

This structural result yields two complementary methods.  \emph{Certified-E7} uses rational nodes and an algebraically isolated correction parameter to cancel $\Cseven$.  Exact root isolation proves existence and uniqueness of the selected parameter in a rational interval, while coprimality with the order-eight chain-defect polynomial proves that the method has classical order exactly seven.  \emph{Accuracy-P40} remains generically sixth order but reduces the complete symmetry-weighted principal-error norm to $9.8\%$ of the LGL--L3 value while satisfying the strict strong-stiff design constraint $\Jstiff<0.40$.

A second conceptual contribution separates three properties that are often conflated: convergence of repeated corrections, absolute stability of the method stopped after a fixed number of sweeps, and fixed-norm behavior for nonnormal stiff operators.  The strong-stiff stage factor $\rho(\Minf)$ predicts asymptotic correction convergence, but a closed affine formula shows that it does not imply $|R_{3,K}(-\infty)|\le1$ at fixed $K$.  This distinction motivates a third-correction safeguard in the far stiff tail without changing the two-correction order construction.

The evidence is organized around these claims.  The theoretical core is the chain
\begin{align*}
 \text{H4 defect direction}&\ \Longrightarrow\
 \text{tree-independent propagation}\\
 &\ \Longrightarrow\ \text{rank-one defect}\\
 &\ \Longrightarrow\ \text{exact seventh-order certificate}.
\end{align*}
Two independent 60-digit nonlinear tests activate branched elementary differentials and verify sixth versus seventh order dynamically.  Fixed-sweep stability calculations support the damping safeguard, while Allen--Cahn and Cahn--Hilliard experiments test whether the strong-stiff objective predicts reduced Newton--Krylov work.

The remainder is organized as follows.  Section~\ref{sec:targets} defines the design metrics.  Section~\ref{sec:method} presents the dense Hermite formula and causal correction method.  Section~\ref{sec:analysis} proves the rank-one theorem and constructs Certified-E7.  Section~\ref{sec:stiff} analyzes stopped stability and matrix-valued stiff limits.  Sections~\ref{sec:designs} and~\ref{sec:numerics} give the reported configurations and numerical evidence, followed by discussion and conclusions.
\section{Design criteria and claim boundaries}
\label{sec:targets}
Let
\begin{equation}
 0=c_0<c_1<\cdots<c_{s-1}<c_s=1.
 \label{eq:nodes}
\end{equation}
The internal nodes are design variables.  After the third-order conditions are imposed, the local endpoint corrector has one free parameter $\beta$.  In the principal case $s=3$ with two corrections,
\begin{equation}
 \theta=(c_1,c_2,\beta).
 \label{eq:theta}
\end{equation}
The stage count, H4 predictor, dense Hermite collocation target, and number of accuracy-defining corrections are fixed.  The comparison therefore isolates the effect of node--sweep co-design rather than changing the architecture.

\subsection{Strong-stiff contraction and stopped output}
For the scalar test equation $u'=\lambda u$, write $z=\lambda\dt$.  Let $M_\beta(z;\bm c)$ be the homogeneous stage-error matrix for one correction and define
\[
 \Minf(\bm c,\beta)=\lim_{|z|\to\infty}M_\beta(z;\bm c).
\]
The local endpoint formula has a separate stiff damping factor $r_\infty(\beta)$.  The design objective is
\begin{equation}
 \boxed{\Jstiff(\bm c,\beta)=
 \max\{\rho(\Minf(\bm c,\beta)),|r_\infty(\beta)|\}.}
 \label{eq:Jstiff}
\end{equation}
The maximum prevents improvement of the stage iteration at the expense of the local endpoint rule, or conversely.  It is nevertheless an iteration-convergence metric, not the absolute stability function of a method stopped after a prescribed number of corrections.

Let $R_{s,K}(z)$ be the endpoint stability function after the predictor and exactly $K$ corrections.  We therefore report separately
\begin{equation}
 R_\infty^{[K]}=\lim_{x\to\infty}R_{s,K}(-x),\qquad
 L_K^- =\sup\{L>0:|R_{s,K}(-x)|\le1\ \text{for }0\le x\le L\}.
 \label{eq:stoppedMetrics}
\end{equation}
These quantities describe the method actually returned after $K$ sweeps.

\subsection{Complete nonlinear principal error}
A stopped method is a B-series method.  If its global order is $p$, the leading local defect is
\begin{equation}
 \mathcal D_{p+1}(u)=\sum_{\tau\in\mathcal T_{p+1}}E_\tau F(\tau)(u),
 \label{eq:BseriesDefect}
\end{equation}
where $\mathcal T_{p+1}$ is the set of unlabeled rooted trees, $F(\tau)$ is the elementary differential, and $E_\tau$ is the numerical-minus-exact coefficient \cite{Butcher2021}.  We measure the complete symmetry-weighted defect by
\begin{equation}
 \mathcal E_{p+1}=(\sigma(\tau)E_\tau)_{\tau\in\mathcal T_{p+1}},\qquad
 \Jtree=\|\mathcal E_{p+1}\|_2.
 \label{eq:treeObjective}
\end{equation}
For $s=3$ and two corrections, $p=6$ generically and $|\mathcal T_7|=48$.  The scalar chain coefficient
\begin{equation}
 \Cseven=[z^7]R_{3,2}(z)-\frac1{7!}
 \label{eq:C7def}
\end{equation}
is not normally a complete nonlinear measure.  The rank-one theorem proves the special identity
\begin{equation}
 \Jtree=\sqrt{886}\,|\Cseven|
 \label{eq:Jtree}
\end{equation}
for the present H4/two-correction architecture.

Accuracy-P40 is a best-found solution of
\begin{equation}
 \min_{0<c_1<c_2<1,\ \beta\ge1/2}\Jtree
 \quad\text{subject to}\quad \Jstiff\le0.40,
 \label{eq:ParetoProblem}
\end{equation}
whereas Certified-E7 fixes rational nodes and imposes the exact algebraic condition $\Cseven=0$.  The P40 optimization is not claimed globally certified.  The seventh-order statement is certified independently by exact algebra.

The intended regime has a broad negative-real spectrum, tight residual tolerances, and naturally available or moderately priced $\Rtwo$.  The method is less attractive when stiffness is mild, one correction already suffices, or $\Rtwo$ must be constructed solely for the time integrator.  This boundary is retained throughout the numerical interpretation.
\section{Hermite collocation and sequential correction}
\label{sec:method}
Let
\[
 0=c_0<c_1<\cdots<c_s=1,
 \qquad
 U_m\approx u(t_n+c_m\dt),
 \qquad U_0=u^n.
\]
Let $\phi_j,\psi_j\in\mathbb P_{2s+1}$ be the cardinal Hermite basis satisfying
\[
 \phi_j(c_\ell)=\delta_{j\ell},\quad \phi_j'(c_\ell)=0,
 \qquad
 \psi_j(c_\ell)=0,\quad \psi_j'(c_\ell)=\delta_{j\ell}.
\]
For the subinterval $[c_{m-1},c_m]$, define
\begin{equation}
 q_{mj}=\int_{c_{m-1}}^{c_m}\phi_j(\tau)\,d\tau,
 \qquad
 \widehat q_{mj}=\int_{c_{m-1}}^{c_m}\psi_j(\tau)\,d\tau.
 \label{eq:weights}
\end{equation}
Because
\[
 \frac{d}{d\tau}\Rone(u(t_n+\tau\dt))
 =\dt\Rtwo(u(t_n+\tau\dt)),
\]
the Hermite increment is
\begin{equation}
 \mathcal Q_m(\bm U)=
 \dt\sum_{j=0}^s q_{mj}\Rone(U_j)
 +\dt^2\sum_{j=0}^s\widehat q_{mj}\Rtwo(U_j).
 \label{eq:Q}
\end{equation}
The dense background equation is
\begin{equation}
 U_m^\star=U_{m-1}^\star+\mathcal Q_m(\bm U^\star),
 \qquad m=1,\ldots,s.
 \label{eq:collocation}
\end{equation}

\subsection{Order of the dense formula}
The interpolation data consist of the values and first derivatives of $\Rone(u(t_n+\tau\dt))$ at $s+1$ nodes.  The Hermite interpolant therefore has degree at most $2s+1$ and its integral is exact for polynomials through degree $2s+1$.

\begin{theorem}[Dense Hermite collocation order]
\label{thm:collocationOrder}
Let $f=\Rone\in C^{2s+2}$ on a neighborhood of the exact solution over one step.  Assume that the block Jacobian of the collocation equations is nonsingular for all sufficiently small $\dt$.  Then the endpoint collocation stage satisfies
\[
 U_s^\star-u(t_n+\dt)=\OO(\dt^{2s+3}),
\]
and the associated one-step method has classical global order $2s+2$ on every fixed finite-dimensional problem.
\end{theorem}
\begin{proof}
Put $F_h(\tau)=f(u(t_n+\tau\dt))$ for $0\le\tau\le1$, and let $H_h$ be its cardinal Hermite interpolant at $c_0,\ldots,c_s$.  The interpolation conditions imply
\[
 H_h(c_j)=F_h(c_j),\qquad H_h'(c_j)=F_h'(c_j),
\]
where $F_h'(c_j)=\dt\,\Rtwo(u(t_n+c_j\dt))$.  Hence the increment obtained by inserting the exact nodal values into \eqref{eq:Q} is precisely
\[
 \dt\int_{c_{m-1}}^{c_m}H_h(\tau)\,d\tau.
\]
For vector-valued $F_h$, apply the scalar Hermite remainder to every bounded linear functional on the phase space.  Uniformly for $\tau\in[0,1]$,
\[
 F_h(\tau)-H_h(\tau)
 =\frac{F_h^{(2s+2)}(\xi_\tau)}{(2s+2)!}
   \omega(\tau),
 \qquad
 \omega(\tau)=\prod_{j=0}^s(\tau-c_j)^2.
\]
Repeated differentiation with respect to the scaled variable gives
$F_h^{(2s+2)}=\OO(\dt^{2s+2})$.  Therefore the exact nodal vector
$\bar U_m=u(t_n+c_m\dt)$ has collocation residual
\[
 \bar U_m-\bar U_{m-1}-\mathcal Q_m(\bar{\bm U})
 =\dt\int_{c_{m-1}}^{c_m}(F_h-H_h)(\tau)\,d\tau
 =\OO(\dt^{2s+3}).
\]
Let $\mathcal F_h(\bm U)=0$ denote the full active-stage collocation system.  Its Jacobian at $\dt=0$ is $E\otimes I$, which is invertible; by assumption the inverse remains bounded for small $\dt$.  The mean-value formula and the inverse-function estimate then give
\[
 \norm{\bm U^\star-\bar{\bm U}}
 \le C\norm{\mathcal F_h(\bar{\bm U})}
 =\OO(\dt^{2s+3}).
\]
In particular the endpoint has local truncation error of order $2s+3$.  The one-step map is locally Lipschitz because it is obtained from a locally unique smooth branch of the stage equations.  The standard local-to-global argument, or a discrete Gronwall estimate over $\OO(\dt^{-1})$ steps, yields global order $2s+2$.
\end{proof}

\subsection{Matrix notation}
Let $E\in\RR^{s\times s}$ be the lower-bidiagonal difference matrix
\[
 E=\begin{pmatrix}
 1&&&\\[-1mm]
 -1&1&&\\
 &\ddots&\ddots&\\
 &&-1&1
 \end{pmatrix},
\]
and partition the Hermite weight matrices into the anchor column and the active columns,
\[
 Q_1=[q_0\mid Q_{1,a}],
 \qquad
 Q_2=[\widehat q_0\mid Q_{2,a}].
\]
On the scalar test equation, the active collocation stages satisfy
\begin{equation}
 A_Q(z)U=b_Q(z)u^n,
 \qquad
 A_Q(z)=E-zQ_{1,a}-z^2Q_{2,a}.
 \label{eq:AQ}
\end{equation}
This form is useful for the finite-sweep stability function and the stiff-limit derivation.  The code constructs the Hermite weights from the cardinal basis rather than from hard-coded tables, so the same implementation supports optimized and classical nodes.

\subsection{Predictor and correction sweep}
Set $\delta_m=(c_m-c_{m-1})\dt$.  The fourth-order sequential Hermite predictor is
\begin{align}
 U_m^{[0]}={}&U_{m-1}^{[0]}
 +\frac{\delta_m}{2}\left[\Rone(U_{m-1}^{[0]})+\Rone(U_m^{[0]})\right]\notag\\
 &+\frac{\delta_m^2}{12}\left[\Rtwo(U_{m-1}^{[0]})-\Rtwo(U_m^{[0]})\right].
 \label{eq:H4}
\end{align}
Each row contains only one new nonlinear unknown.  In contrast, solving the dense collocation system directly would couple all active stages.

Start from the general two-derivative endpoint increment
\begin{equation}
 \delta\{a\Rone(v)+b\Rone(w)\}
 +\delta^2\{c\Rtwo(v)+d\Rtwo(w)\}.
 \label{eq:generalEndpoint}
\end{equation}
Imposing consistency and third-order conditions gives
\[
 a+b=1,
 \qquad
 b+c+d=\frac12,
 \qquad
 \frac b2+d=\frac16.
\]
Writing $b=\beta$ leaves the one-parameter family
\begin{align}
 \mathcal P_{\beta,\delta}(v,w)={}&
 \delta\{(1-\beta)\Rone(v)+\beta\Rone(w)\}\notag\\
 &+\delta^2\left\{\left(\frac13-\frac\beta2\right)\Rtwo(v)
 +\left(\frac16-\frac\beta2\right)\Rtwo(w)\right\}.
 \label{eq:Pbeta}
\end{align}

Given $\bm U^{[k]}$, one correction sweep is
\begin{align}
 U_m^{[k+1]}={}&U_{m-1}^{[k+1]}
 +\mathcal P_{\beta,\delta_m}(U_{m-1}^{[k+1]},U_m^{[k+1]})\notag\\
 &-\mathcal P_{\beta,\delta_m}(U_{m-1}^{[k]},U_m^{[k]})
 +\mathcal Q_m(\bm U^{[k]}).
 \label{eq:correction}
\end{align}
The nodes are updated in the causal order $U_1\to\cdots\to U_s$.  At a fixed point the two local increments cancel and \eqref{eq:collocation} remains.

\begin{proposition}[Sequential solvability and fixed point]
\label{prop:solvability}
Assume $\Rone,\Rtwo\in C^1$ near the exact step solution.  For sufficiently small $\dt$, every predictor and correction row has a locally unique solution obtained in causal order.  Moreover, a stage vector is a fixed point of the correction iteration if and only if it solves the dense Hermite collocation equations.
\end{proposition}
\begin{proof}
For a predictor row, with the left stage fixed, write the residual as
\[
 \mathcal G_m(w)=w-v-\frac{\delta_m}{2}\{\Rone(v)+\Rone(w)\}
 -\frac{\delta_m^2}{12}\{\Rtwo(v)-\Rtwo(w)\}.
\]
Its derivative with respect to the new stage is
\[
 D_w\mathcal G_m(w)=I-\frac{\delta_m}{2}\Rone_u(w)
 +\frac{\delta_m^2}{12}\Rtwo_u(w)=I+\OO(\dt).
\]
The same calculation for a correction row gives
\[
 I-\beta\delta_m\Rone_u(w)
 -\left(\frac16-\frac\beta2\right)\delta_m^2\Rtwo_u(w)
 =I+\OO(\dt).
\]
Both derivatives are invertible for sufficiently small $\dt$, uniformly on a fixed neighborhood.  The implicit-function theorem therefore gives a unique local solution for row $m$ once rows $0,\ldots,m-1$ are known.  Induction over $m$ proves causal solvability of a full predictor or correction sweep.

If $\bm U^{[k+1]}=\bm U^{[k]}=\bm U$, the two occurrences of
$\mathcal P_{\beta,\delta_m}$ in \eqref{eq:correction} cancel, leaving
$U_m=U_{m-1}+\mathcal Q_m(\bm U)$ for every $m$.  Conversely, a solution of these collocation equations makes the correction residual vanish when used on both sides, and hence is a fixed point.
\end{proof}

\begin{algorithm}[t]
\caption{Stopped two-derivative Hermite correction with order safeguards}
\label{alg:method}
\begin{algorithmic}[1]
\Require $u^n$, nodes $\bm c$, parameter $\beta$, minimum and maximum corrections $K_{\min}\le K_{\max}$, tolerances $\varepsilon_{\rm corr}$ and $\varepsilon_{\rm row}$
\State Compute the H4 predictor \eqref{eq:H4} sequentially, solving every row to tolerance $\varepsilon_{\rm row}$.
\For{$k=0,\ldots,K_{\max}-1$}
 \State Evaluate the scaled dense-collocation residual.
 \If{$k\ge K_{\min}$ and the residual is below $\varepsilon_{\rm corr}$} \State stop; \EndIf
 \For{$m=1,\ldots,s$}
   \State solve the single-state nonlinear equation \eqref{eq:correction} to tolerance $\varepsilon_{\rm row}$;
 \EndFor
\EndFor
\State Set $u^{n+1}=U_s$.
\end{algorithmic}
\end{algorithm}

\subsection{Residual stopping and computational work}
For an active stage vector $\bm U=(U_1,\ldots,U_s)$, define the row residuals
\[
 \mathcal F_m(\bm U)=
 U_m-U_{m-1}-\mathcal Q_m(\bm U).
\]
The implementation stops when
\begin{equation}
 \frac{\max_m\norm{\mathcal F_m(\bm U)}}
 {1+\max_m\norm{U_m}}
 \le\varepsilon_{\rm corr}.
 \label{eq:residualStop}
\end{equation}
All reported sweep counts mean corrections performed after the H4 predictor.  The predictor itself is never counted as a correction sweep.  For Certified-E7 we require $K_{\min}=2$ unless the omitted corrections are independently shown to be below the order-eight local error scale.

\begin{proposition}[Inexact row solves and early stopping]
\label{prop:inexact}
Fix the number of stages and completed corrections.  Suppose the stopped one-step map is uniformly locally Lipschitz and all causal row inverse Jacobians are uniformly bounded for $0<\dt\le\dt_0$.  If the sum of normalized predictor, row, and stopping residuals in macrostep $n$ is at most $\eta_n$, then the computed endpoint differs from the exactly solved stopped method by at most $C\eta_n$.  On a fixed time interval the accumulated perturbation is $\OO(\max_n\eta_n/\dt)$.  Hence an order-$p$ stopped method retains order $p$ whenever
\begin{equation}
 \eta_n=\OO(\dt^{p+1}).
 \label{eq:inexactScaling}
\end{equation}
In particular, Certified-E7 requires two completed corrections and local row/collocation residuals of order $\OO(\dt^8)$.
\end{proposition}
\begin{proof}
For each causal row, the mean-value identity and the bounded inverse-Jacobian assumption bound the new-stage perturbation by the row residual plus perturbations of already available stages and old-sweep data.  Induction over the fixed number of rows and sweeps gives a macrostep endpoint perturbation $C\eta_n$.  If $e_n$ is the accumulated perturbation, local Lipschitz continuity yields
\[
 \|e_{n+1}\|\le(1+L\dt)\|e_n\|+C\eta_n.
\]
Discrete Gronwall over $N=\OO(\dt^{-1})$ steps gives $\max_n\|e_n\|\le C_T\max_n\eta_n/\dt$, proving the stated scaling.
\end{proof}

The constants can depend on the spatial grid through row Jacobians, resolvent bounds, and nonlinear Lipschitz constants.  The result establishes classical order on each fixed semidiscretization; a mesh-uniform stiff convergence theorem would require additional regularity and resolvent assumptions.  In the sparse phase-field implementation, Newton iteration solves each single-state row and GMRES solves the Newton systems with reused ILU factors.
\section{Accuracy and stability analysis}
\label{sec:analysis}
\subsection{Endpoint stability family}
For the scalar equation $u'=\lambda u$, let $z=\lambda\delta$.  Applying \eqref{eq:Pbeta} to one endpoint step gives
\begin{equation}
 R_\beta(z)=
 \frac{1+(1-\beta)z+(\frac13-\frac\beta2)z^2}
 {1-\beta z+(\frac\beta2-\frac16)z^2}.
 \label{eq:Rbeta}
\end{equation}
The parameter changes both the implicit first-derivative weight and the two second-derivative weights.  In particular, $\beta=2/3$ gives the classical L3 rule, whereas smaller admissible values trade exact local stiff decay for a different correction preconditioner.

\begin{theorem}[Exact endpoint classification]
\label{thm:Astable}
The family \eqref{eq:Rbeta} is $A$-stable if and only if $\beta\ge1/2$.  For $\beta>1/2$,
\begin{equation}
 r_\infty(\beta)=\lim_{|z|\to\infty}R_\beta(z)
 =\frac{2-3\beta}{3\beta-1},
 \label{eq:rinf}
\end{equation}
and the unique $L$-stable member is $\beta=2/3$.
\end{theorem}
\begin{proof}
Write $R_\beta=N_\beta/D_\beta$.  For $z=iy$, direct expansion gives the exact identity
\begin{equation}
 |D_\beta(iy)|^2-|N_\beta(iy)|^2
 =\frac{2\beta-1}{12}y^4.
 \label{eq:imaginaryIdentity}
\end{equation}
If $\beta<1/2$, the right-hand side is negative for every $y\ne0$, so
$|R_\beta(iy)|>1$ and $A$-stability is impossible.

Now let $\beta\ge1/2$ and put $a=\beta/2-1/6>0$.  The denominator is
$D_\beta(z)=az^2-\beta z+1$.  If its roots are real, their sum is $\beta/a>0$ and their product is $1/a>0$, so both roots are positive.  If they are a complex-conjugate pair, each has real part $\beta/(2a)>0$.  Thus $D_\beta$ has no zero in the closed left half-plane.  Identity \eqref{eq:imaginaryIdentity} gives $|R_\beta(iy)|\le1$ on its finite boundary.  Moreover,
\[
 \lim_{|z|\to\infty}|R_\beta(z)|
 =\left|\frac{1/3-\beta/2}{\beta/2-1/6}\right|
 =\left|\frac{2-3\beta}{3\beta-1}\right|\le1
 \qquad(\beta\ge1/2).
\]
Apply the maximum-modulus principle on left half-disks and pass to the limit in the radius; this proves $|R_\beta(z)|\le1$ whenever $\operatorname{Re}z\le0$.

The ratio of the quadratic coefficients gives \eqref{eq:rinf}.  An $A$-stable rational method is $L$-stable precisely when this limit is zero, which occurs only for $\beta=2/3$.
\end{proof}

\begin{remark}
The classification concerns the local endpoint formula.  It does not determine the convergence of the complete Hermite correction iteration, which also depends on the node matrices.  This distinction is precisely why $\beta$ and the nodes are designed together.
\end{remark}

\subsection{Finite-sweep lower bound and generic sharpness}
Throughout this subsection the phase-space dimension and number of stages are fixed.  Assume that $\Rone$ and $\Rtwo$ are sufficiently smooth and that the inverse Jacobians of the causal predictor and correction rows remain bounded for $0<\dt\le\dt_0$.

\begin{theorem}[Stopped-method lower bound]
\label{thm:order}
With an H4 predictor and $K$ completed correction sweeps, the stopped method has classical global order at least
\begin{equation}
 p_{s,K}^{\rm lb}=\min\{4+K,2s+2\}.
 \label{eq:order}
\end{equation}
If a defect coefficient at the first potentially nonzero order is not identically zero in the admissible design parameters, then the lower bound is attained on an open dense subset of every connected nonsingular parameter region.  Exceptional analytic zero sets may produce superconvergence.
\end{theorem}
\begin{proof}
Let $\bm U^\star$ be the locally unique collocation stage vector and set
$e_m^{[k]}=U_m^{[k]}-U_m^\star$.  The composite H4 predictor is fourth order on each fixed stage, hence
\[
 \max_m\norm{e_m^{[0]}}\le C\dt^5.
\]
Subtract the collocation row from correction row $m$.  Since
$U_m^\star-U_{m-1}^\star=\mathcal Q_m(\bm U^\star)$, the exact error equation is
\begin{align*}
 e_m^{[k+1]}={}&e_{m-1}^{[k+1]}
 +\bigl[\mathcal P_m(U_{m-1}^{[k+1]},U_m^{[k+1]})
       -\mathcal P_m(U_{m-1}^\star,U_m^\star)\bigr]\\
 &-\bigl[\mathcal P_m(U_{m-1}^{[k]},U_m^{[k]})
       -\mathcal P_m(U_{m-1}^\star,U_m^\star)\bigr]
 +\bigl[\mathcal Q_m(\bm U^{[k]})-\mathcal Q_m(\bm U^\star)\bigr].
\end{align*}
The first derivatives of $\mathcal P_m$ are $\OO(\dt)$ and those of
$\mathcal Q_m$ are $\OO(\dt)$ on a fixed finite-dimensional neighborhood.  Move the term involving the new right stage to the left.  Its coefficient is $I+\OO(\dt)$ and has a uniformly bounded inverse.  Causal induction in $m$ then gives
\begin{equation}
 \max_m\norm{e_m^{[k+1]}}
 \le C\dt\max_m\norm{e_m^{[k]}},
 \label{eq:sweepGain}
\end{equation}
for all sufficiently small $\dt$.  Consequently
$\max_m\norm{e_m^{[K]}}=\OO(\dt^{K+5})$.

The exact endpoint differs from the collocation endpoint by
$\OO(\dt^{2s+3})$ by \cref{thm:collocationOrder}.  Therefore the stopped method has local error
\[
 \OO(\dt^{K+5})+\OO(\dt^{2s+3})
 =\OO\bigl(\dt^{\min\{K+5,2s+3\}}\bigr),
\]
and hence global order at least \eqref{eq:order}.

For generic sharpness, restrict attention to a connected region with distinct nodes and nonsingular row Jacobians.  At every fixed tree order, the B-series coefficients are obtained from finitely many additions, products, and inversions of analytic functions of the nodes and $\beta$.  They are therefore real analytic.  If one coefficient at the first candidate order is not the zero function, its zero set has empty interior.  The complement is open and dense, and on that complement the first candidate defect is nonzero, so the lower bound is sharp.
\end{proof}

\begin{corollary}[Generic sixth order for the principal architecture]
\label{cor:generic6}
For $s=3$ and $K=2$, the stopped method is generically sixth order.  At the rational design
\begin{equation}
 (c_1,c_2,\beta)=\left(\frac14,\frac34,\frac23\right),
 \qquad
 \Cseven=-\frac{5483}{19025362944}\ne0.
 \label{eq:genericWitness}
\end{equation}
Certified-E7 belongs to the exceptional algebraic set $\Cseven=0$.
\end{corollary}
\begin{proof}
For $s=3$ and $K=2$, \cref{thm:order} gives global order at least six; the first candidate local defect has order seven.  The exact rational formal-series recurrence in Appendix~\ref{app:audit} evaluates the order-seven chain component at the displayed rational point and gives \eqref{eq:genericWitness}.  Hence that coefficient is not identically zero as a function of the design parameters.  Generic sharpness in \cref{thm:order} now proves sixth order on an open dense set.  The exceptional construction is established in \cref{thm:E7}.
\end{proof}

\subsection{Rooted-tree notation}
A rooted tree is written $\tau=[\tau_1,\ldots,\tau_m]$, with the one-node tree denoted by $\bullet$.  Let $|\tau|$ be its order, $\sigma(\tau)$ its symmetry factor, and
\[
 \mathcal U(\tau)=[\tau]
\]
the unary-grafting operator.  We use the B-series normalization
\begin{equation}
 B(a,hf,u)=u+\sum_{\tau\ne\emptyset}h^{|\tau|}a(\tau)F(\tau)(u).
 \label{eq:Bseries}
\end{equation}
The number of unlabeled rooted trees of orders one through seven is
\[
 1,\ 1,\ 2,\ 4,\ 9,\ 20,\ 48.
\]
For a B-series stage $Y=B(a,hf,u)$ and a tree $\tau=[\tau_1,\ldots,\tau_m]$, let the distinct child multiplicities be $\mu_1,\ldots,\mu_r$.  The composition coefficients used below are
\begin{align}
 a_{f(Y)}(\tau)&=
 \frac{\prod_{i=1}^m a_Y(\tau_i)}{\prod_{j=1}^r\mu_j!},
 \label{eq:fComposition}\\
 a_{g(Y)}(\tau)&=
 \sum_{\theta\in\operatorname{supp}(\tau)}
 \frac{a_{f(Y)}(\theta)
 \prod_{\tau_i\ne\text{one selected }\theta}a_Y(\tau_i)}
 {\prod_{\vartheta}(\mu_\vartheta-\mathbf1_{\vartheta=\theta})!},
 \qquad g=f'f,
 \label{eq:gComposition}
\end{align}
and the exact-flow coefficients satisfy
\begin{equation}
 a_{\rm ex}(\tau)=\frac1{|\tau|}
 a_{f(B(a_{\rm ex}))}(\tau).
 \label{eq:exactTreeRecursion}
\end{equation}
Applying \eqref{eq:fComposition}--\eqref{eq:exactTreeRecursion} to every predictor and correction stage gives a deterministic coefficient recursion for all trees through order eight.

\subsection{Complete order-seven defect}
The main structural fact is that the H4 predictor has a one-directional principal defect and that the leading action of one correction is unary grafting.  To make the argument explicit, let
\[
 d_{m,r}^{[k]}\in\mathbb R^{\mathcal T_r}
\]
denote the order-$r$ B-series coefficient defect at stage $m$ after sweep $k$, measured relative to the dense Hermite collocation stage.  The anchor stage satisfies $d_{0,r}^{[k]}=0$.

\begin{lemma}[H4 principal-defect direction]
\label{lem:predictorDefect}
For every active stage $m$ there exists a scalar $\eta_m$, depending only on the node increments, such that the first predictor defect relative to the dense collocation stage is
\begin{equation}
 d_{m,5}^{[0]}=\eta_m\,a_{\rm ex}|_{\mathcal T_5}.
 \label{eq:predictorDirection}
\end{equation}
All defects of orders below five vanish.
\end{lemma}
\begin{proof}
Consider first one H4 substep of scaled length $d=c_m-c_{m-1}$.  Applied to the exact solution, its quadrature error is the linear functional
\[
 \mathcal L_{d,h}(F)
 =h\int_0^dF(\tau)\,d\tau
 -\frac{dh}{2}\{F(0)+F(d)\}
 -\frac{d^2h}{12}\{F'(0)-F'(d)\},
\]
with $F(\tau)=f(u(t_{m-1}+\tau h))$.  The functional annihilates polynomials of degree at most three.  Taylor expansion at the left endpoint therefore gives
\[
 \mathcal L_{d,h}(F)=\mu_d h^5u^{(5)}(t_{m-1})+\OO(h^6),
\]
where $\mu_d$ is a scalar moment.  In particular, no elementary differential of order below five occurs.

In the B-series normalization \eqref{eq:Bseries}, the time derivative
$u^{(5)}$ is
\[
 u^{(5)}=5!\sum_{\theta\in\mathcal T_5}a_{\rm ex}(\theta)F(\theta).
\]
Thus the local order-five defect of every H4 substep is a scalar multiple of the same vector
$a_{\rm ex}|_{\mathcal T_5}$.  When a defect created on an earlier subinterval is transported through a subsequent exact or fourth-order substep, the derivative of the transport map is $I+\OO(h)$.  Its order-five homogeneous part is therefore unchanged, while all modifications begin at order six.  Summing the transported subinterval defects and inducting over $m$ produces a stage-dependent scalar $\eta_m$ multiplying the common direction, which proves \eqref{eq:predictorDirection}.
\end{proof}

\begin{proposition}[Stage-wise principal-defect propagation]
\label{prop:stagePropagation}
Assume that at sweep $k$ the first stage defects occur at order $r\ge5$ and share one homogeneous tree direction,
\begin{equation}
 d_{m,r}^{[k]}=\lambda_m^{[k]}e_r,
 \qquad m=0,\ldots,s,
 \label{eq:commonDirection}
\end{equation}
with $\lambda_0^{[k]}=0$.  Then one correction cancels the order-$r$ defect.  Its first possible defect has order $r+1$ and has the form
\begin{equation}
 d_{m,r+1}^{[k+1]}
 =\lambda_m^{[k+1]}\mathcal U(e_r),
 \qquad
 \bm\lambda^{[k+1]}=T(\bm c,\beta)\bm\lambda^{[k]},
 \label{eq:stagePropagation}
\end{equation}
where $T$ is a lower-triangular scalar stage matrix independent of the rooted tree contained in $e_r$.
\end{proposition}
\begin{proof}
Write the old stages as formal B-series perturbations of the collocation stages,
\[
 U_j^{[k]}=U_j^\star+h^r\lambda_j^{[k]}e_r+\OO(h^{r+1}).
\]
The Fr\'echet expansion of the value derivative is
\[
 f(U_j^{[k]})-f(U_j^\star)
 =h^r\lambda_j^{[k]}f'(u)e_r+\OO(h^{r+1}).
\]
For a homogeneous B-series vector, multiplication by $f'(u)$ grafts one new root above every represented tree; hence the order-$r+1$ coefficient direction is exactly $\mathcal U(e_r)$.

The second-derivative field $g=f'f$ is smooth, but every occurrence of $g$ in the method is multiplied by $h^2$.  Its first variation caused by an order-$r$ stage perturbation therefore contributes only at order $r+2$.  Quadratic products of stage perturbations also start beyond order $r+1$.  Thus only the $h f$ terms can contribute at the first new order.

At order $r$, the dense residual and the subtracted old preconditioner carry identical stage values, so the correction equation removes the old order-$r$ defect.  At order $r+1$, collect the coefficients of $\mathcal U(e_r)$.  For correction row $m$ one obtains
\begin{equation}
 \lambda_m^{[k+1]}
 =\lambda_{m-1}^{[k+1]}
 +\sum_{j=0}^s q_{mj}^{(1)}\lambda_j^{[k]}
 -\Delta c_m\bigl((1-\beta)\lambda_{m-1}^{[k]}+\beta\lambda_m^{[k]}\bigr),
 \label{eq:stageScalarMain}
\end{equation}
where $q_{mj}^{(1)}$ denotes the value-data Hermite weight.  Terms involving the unknown new right-stage perturbation occur inside $h f(U_m^{[k+1]})$ and therefore first affect order $r+2$; equivalently, the implicit row factor is $I+\OO(h)$ and acts as the identity at order $r+1$.

Equation \eqref{eq:stageScalarMain} is a causal scalar forward substitution.  It defines a lower-triangular matrix $T(\bm c,\beta)$ and contains no tree-specific coefficient.  Hence every component of $e_r$ is propagated with the same scalar stage amplitudes and the sole new tree operation is unary grafting.
\end{proof}

\begin{theorem}[Rank-one order-seven defect]
\label{thm:rankone}
For $s=3$, the H4 predictor, and exactly two corrections, the order-seven local B-series defect satisfies, for every admissible node set and every nonsingular parameter $\beta$,
\begin{align}
 E_\tau&=0,
 &&\tau\notin\mathcal U^2(\mathcal T_5),
 \label{eq:rankzero}\\
 E_{\mathcal U^2(\theta)}&=120a_{\rm ex}(\theta)\Cseven,
 &&\theta\in\mathcal T_5.
 \label{eq:rankvalues}
\end{align}
Consequently,
\begin{equation}
 \norm{(E_\tau)_{\tau\in\mathcal T_7}}_2
 =\sqrt{90}\,|\Cseven|,
 \qquad
 \norm{(\sigma(\tau)E_\tau)_{\tau\in\mathcal T_7}}_2
 =\sqrt{886}\,|\Cseven|,
 \label{eq:treenorm}
\end{equation}
and $\Cseven=0$ is equivalent to all 48 nonlinear seventh-order conditions.
\end{theorem}
\begin{proof}
By \cref{lem:predictorDefect}, the first predictor defects have order five and are of the form
$d_{m,5}^{[0]}=\eta_m a_{\rm ex}|_{\mathcal T_5}$.  Apply
\cref{prop:stagePropagation} once.  The order-five defect is cancelled and the order-six defect is a scalar stage amplitude multiplying
$\mathcal U(a_{\rm ex}|_{\mathcal T_5})$.  Apply the proposition a second time.  The order-six defect is cancelled and the order-seven endpoint defect is
\[
 d_{s,7}^{[2]}=\gamma(\bm c,\beta)
 \mathcal U^2(a_{\rm ex}|_{\mathcal T_5})
\]
for one scalar $\gamma$.  This immediately proves \eqref{eq:rankzero} and shows that the image of the complete order-seven defect map is one-dimensional.

It remains to identify $\gamma$.  Let $\theta_{\rm ch}$ be the order-five chain tree.  Its twice-grafted image is the order-seven chain.  The exact-flow coefficient of $\theta_{\rm ch}$ is $a_{\rm ex}(\theta_{\rm ch})=1/5!=1/120$.  By definition, the defect coefficient of the order-seven chain is $\Cseven$.  Therefore
\[
 \Cseven=\gamma a_{\rm ex}(\theta_{\rm ch})=\frac{\gamma}{120},
 \qquad\text{so}\qquad \gamma=120\Cseven.
\]
Substitution gives \eqref{eq:rankvalues} for every $\theta\in\mathcal T_5$.

The nine supported multipliers have squared sums $90$ and, after symmetry weighting, $886$, which gives \eqref{eq:treenorm}.  Exact arithmetic independently enumerates all 48 trees and checks these identities; the parameter-independent conclusion follows from the stage-wise argument above.
\end{proof}

\subsection{Algebraic certification of exact seventh order}
Choose the rational nodes
\begin{equation}
 \bm c_E=\left(0,\frac7{20},\frac{37}{50},1\right).
 \label{eq:Enodes}
\end{equation}
For these nodes, exact formal-series inversion of \eqref{eq:auditIteration} gives
\begin{equation}
 \Cseven(\bm c_E,\beta)=
 \frac{p_7(\beta)}{612698688000000000000000000000},
 \label{eq:C7rational}
\end{equation}
where
\begin{align}
 p_7(\beta)={}&-1783651945616635920000000\,\beta^2\notag\\
 &+2149402403417268979914972\,\beta
 -647191260859839121135681.
 \label{eq:p7explicit}
\end{align}
The denominator in \eqref{eq:C7rational} is positive.  Exact substitution gives
\begin{align}
 p_7\!\left(\frac7{12}\right)&=-\frac{930937749612483555842}{3}<0,\notag\\
 p_7\!\left(\frac35\right)&=\frac{1677403842666678066511}{5}>0,\notag\\
 p_7\!\left(\frac58\right)&=-\frac{1107599961088829877647}{2}<0.
 \label{eq:simpleIsolation}
\end{align}
Because $p_7$ is quadratic, these signs isolate its two roots as
\begin{equation}
 \frac7{12}<\beta_-<\frac35<\beta_E<\frac58.
 \label{eq:Einterval}
\end{equation}
Its discriminant is positive but not a square in $\mathbb Z$.  Hence both roots are algebraic irrational numbers: no rational fraction can represent either root exactly.  The exact method is specified without a decimal expansion by
\begin{equation}
 \beta_E:=\text{the unique root of }p_7(\beta)=0
 \text{ in }\left(\frac35,\frac58\right).
 \label{eq:betaExactDefinition}
\end{equation}
A short decimal, $\beta_E\approx0.61647415$, is used only in tables.

To determine whether the method could have order eight, the same recurrence gives the order-eight chain defect
\begin{equation}
 C_{8,\mathrm{ch}}(\beta)=
 \frac{p_8(\beta)}{58714915271040000000000000000000000000},
 \label{eq:C8rational}
\end{equation}
with primitive numerator
\begin{align}
 p_8(\beta)={}&-127422974006480605680124800000000\,\beta^3\notag\\
 &+46198250277252396819443304013776\,\beta^2\notag\\
 &+96232479174981416895187909363404\,\beta\notag\\
 &-45991604262471842675518344262921.
 \label{eq:p8explicit}
\end{align}
Reduction of \eqref{eq:p7explicit} and \eqref{eq:p8explicit} modulo $11$ gives
\begin{equation}
 \overline p_7=-\beta^2+4\beta-4,
 \qquad
 \overline p_8=-\beta^3-2\beta^2+5\beta+2.
 \label{eq:modPolys}
\end{equation}
The Euclidean algorithm in $\mathbb F_{11}[\beta]$ is
\begin{align}
 \overline p_8&=(\beta-5)\overline p_7+(4-4\beta),\notag\\
 \overline p_7&=(3\beta+2)(4-4\beta)-1.
 \label{eq:modEuclid}
\end{align}
Thus $\gcd(\overline p_7,\overline p_8)=1$.  Since the reductions preserve the polynomial degrees, $p_7$ and $p_8$ are coprime over $\mathbb Q$.

\begin{theorem}[Certified-E7 has exact classical order seven]
\label{thm:E7}
Let $\beta_E$ be defined by \eqref{eq:betaExactDefinition}.  The method with nodes \eqref{eq:Enodes}, the H4 predictor, and exactly two completed corrections has classical global order exactly seven for sufficiently smooth fixed-dimensional problems with locally unique stage solves.
\end{theorem}
\begin{proof}
The exact definition of $p_7$ gives $\Cseven(\bm c_E,\beta_E)=0$.  By
\cref{thm:rankone}, every order-seven rooted-tree defect then vanishes, so the local truncation error is $\OO(\dt^8)$ and the global order is at least seven.

The quadratic $p_7$ is primitive and has nonsquare discriminant, so it is irreducible over $\mathbb Q$ and is the minimal polynomial of $\beta_E$.  Suppose, for contradiction, that the method had order at least eight.  Then every order-eight local defect would vanish, in particular the chain defect, and hence $p_8(\beta_E)=0$.  The minimal-polynomial property would imply $p_7\mid p_8$ in $\mathbb Q[\beta]$.  This contradicts the coprimality established by \eqref{eq:modPolys}--\eqref{eq:modEuclid}.  Therefore at least one order-eight defect is nonzero and the classical global order is exactly seven.
\end{proof}

The quadratic has a second root in $(7/12,3/5)$.  Both roots are seventh-order, but exact-formula evaluation gives $\Jstiff=0.6012321\ldots$ for the smaller root and $0.4723251\ldots$ for $\beta_E$; hence the larger root is used.
\section{Strong-stiff design and stopped stability}
\label{sec:stiff}
This section connects the design variables to convergence of the correction iteration and to the output of the method stopped after a finite number of sweeps.  The distinction between formal order and behavior on stiff modes is classical \cite{ProtheroRobinson1974,HairerWanner1996}; here it must additionally be resolved at a prescribed, finite correction count.

\subsection{Finite-parameter correction matrix}
Let $D=\operatorname{diag}(\Delta c_1,\ldots,\Delta c_s)$, where $\Delta c_m=c_m-c_{m-1}$.  The first- and second-derivative parts of the causal sweep generate lower-bidiagonal matrices $A_\beta$ and $B_\beta$.  The leading second-derivative matrix is
\begin{equation}
 B_\beta=\begin{pmatrix}
 \frac\beta2-\frac16&&&\\
 \frac\beta2-\frac13&\frac\beta2-\frac16&&\\
 &\ddots&\ddots&\\
 &&\frac\beta2-\frac13&\frac\beta2-\frac16
 \end{pmatrix}.
 \label{eq:Bbeta}
\end{equation}
The active preconditioner on the scalar test equation has the form
\begin{equation}
 P_\beta(z)=E-zDA_\beta+z^2D^2B_\beta,
 \label{eq:Pmatrix}
\end{equation}
with an anchor term on the right-hand side.  One correction can be written
\begin{equation}
 U^{[k+1]}=M_\beta(z)U^{[k]}+r_\beta(z)u^n,
 \qquad
 M_\beta(z)=I-P_\beta(z)^{-1}A_Q(z).
 \label{eq:finiteIteration}
\end{equation}
The stopped scalar stability function after $K$ corrections is obtained by iterating \eqref{eq:finiteIteration} and projecting onto the endpoint stage:
\begin{equation}
 R_{s,K}(z)=e_s^T\left[
 M_\beta(z)^KU^{[0]}(z)
 +\sum_{\ell=0}^{K-1}M_\beta(z)^\ell r_\beta(z)
 \right].
 \label{eq:finiteStability}
\end{equation}
This formula is used by the formal-series and rooted-tree code; it is also the link between $\Cseven$ and the finite-sweep method.

\begin{theorem}[Strong-stiff output of a stopped method]
\label{thm:stoppedLimit}
Let $\widehat q_0$ be the anchor column of $Q_2$, define the affine vector
\[
 b_\infty=(D^2B_\beta)^{-1}\widehat q_0,
\]
and let $\bm1$ be the vector of ones.  Then
\begin{equation}
 U_\infty^{[0]}=\bm1,
 \qquad
 R_\infty^{[K]}
 =e_s^T\left[\Minf^K\bm1+
 \sum_{\ell=0}^{K-1}\Minf^\ell b_\infty\right].
 \label{eq:stoppedLimit}
\end{equation}
Consequently, $\rho(\Minf)<1$ controls convergence as $K\to\infty$ but does not imply $|R_\infty^{[K]}|\le1$ for a fixed finite $K$.
\end{theorem}
\begin{proof}
On the scalar test equation, divide the H4 predictor equations by $z^2$ and let $z\to-\infty$.  Only the second-derivative terms remain.  In causal row form they state that each limiting stage equals the previous one; the anchor is one, hence $U_\infty^{[0]}=\bm1$.

The finite-parameter correction is the affine map
$U^{[k+1]}=M_\beta(z)U^{[k]}+b_\beta(z)$.  The matrix limit is $\Minf$ and, by the same leading-order division, the affine term tends to
$b_\infty=(D^2B_\beta)^{-1}\widehat q_0$.  Iterating the limiting affine recurrence gives
\[
 U_\infty^{[K]}=\Minf^K\bm1+
 \sum_{\ell=0}^{K-1}\Minf^\ell b_\infty.
\]
The endpoint projection yields \eqref{eq:stoppedLimit}.  Even when $\rho(\Minf)<1$, a finite power and the accumulated affine forcing can have endpoint magnitude larger than one; only the error between successive iterates is governed asymptotically by the spectral radius.
\end{proof}

\begin{table}[t]
\centering
\caption{Fixed-sweep output stability.  $L_K^-$ is the connected negative-real stability radius estimated on a logarithmic grid; $\infty$ means that no violation was detected up to $-10^5$ and is accompanied by $|R_\infty^{[K]}|<1$.}
\label{tab:stoppedStability}
\small
\begin{tabular}{lrrrr}
\toprule
Method & $R_\infty^{[2]}$ & $L_2^-$ & $R_\infty^{[3]}$ & $L_3^-$\\
\midrule
LGL--L3 & $2.5432$ & $94.1$ & $2.0301$ & $115.7$\\
Accuracy-P40 & $1.1995$ & $340.6$ & $0.8420$ & $\infty$\\
Certified-E7 & $1.3808$ & $225.1$ & $0.5135$ & $\infty$\\
\bottomrule
\end{tabular}
\end{table}

\begin{figure}[t]
\centering
\includegraphics[width=0.84\linewidth]{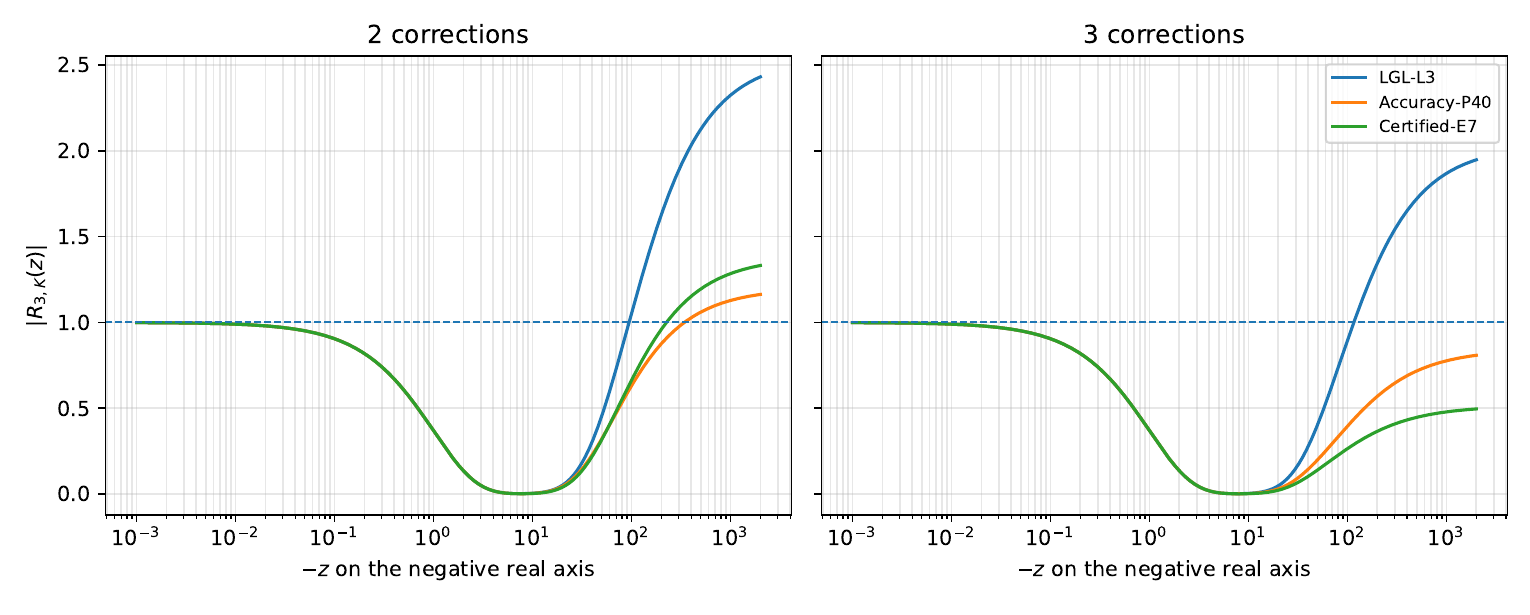}
\caption{Absolute stability on the negative real axis for the methods stopped after two and three corrections.  The curves display the complete interval data summarized in \cref{tab:stoppedStability}; the third correction restores far-stiff damping for both co-designed methods.}
\label{fig:stoppedStability}
\end{figure}

\subsection{Strong-stiff limit}
\begin{theorem}[Strong-stiff correction matrix]
\label{thm:Minf}
For $\beta>1/2$, the correction matrix satisfies
\begin{equation}
 M_\beta(z)\longrightarrow
 \Minf(\bm c,\beta)
 :=I+[D^2B_\beta]^{-1}Q_{2,a},
 \qquad |z|\to\infty.
 \label{eq:Minf}
\end{equation}
The limiting homogeneous correction error is asymptotically contractive if and only if $\rho(\Minf)<1$.
\end{theorem}
\begin{proof}
The active collocation matrix and the lower-triangular sweep matrix have expansions
\[
 A_Q(z)=-z^2Q_{2,a}-zQ_{1,a}+E,
 \qquad
 P_\beta(z)=z^2D^2B_\beta+zDB_\beta^{(1)}+E,
\]
where $B_\beta^{(1)}$ is the first-derivative sweep matrix.  For $\beta>1/2$, the diagonal of $B_\beta$ is
$\beta/2-1/6>0$, so $D^2B_\beta$ is nonsingular.  Factor the leading term:
\[
 P_\beta(z)=z^2D^2B_\beta
 \left[I+z^{-1}(D^2B_\beta)^{-1}DB_\beta^{(1)}+\OO(z^{-2})\right].
\]
A Neumann expansion gives
\[
 P_\beta(z)^{-1}=z^{-2}(D^2B_\beta)^{-1}+\OO(z^{-3}).
\]
Since $M_\beta(z)=I-P_\beta(z)^{-1}A_Q(z)$, multiplication yields
\[
 M_\beta(z)=I+(D^2B_\beta)^{-1}Q_{2,a}+\OO(z^{-1}),
\]
which proves \eqref{eq:Minf}.  The limiting homogeneous iteration is
$e^{[k+1]}=\Minf e^{[k]}$; in finite dimension its powers converge to zero exactly when $\rho(\Minf)<1$.
\end{proof}

The matrix \eqref{eq:Minf} explains why local $L$-stability is insufficient.  The LGL--L3 baseline has $r_\infty=0$, but its complete stage factor is $\rho(\Minf)=0.519109$.  Moving the nodes and changing $\beta$ can reduce the stage factor even when the endpoint rule is no longer exactly $L$-stable.

\section{Reported configurations and operational interpretation}
\label{sec:designs}
The main comparison uses the classical LGL--L3 baseline and two new configurations.  Their architecture is identical: three subintervals, an H4 predictor, and two accuracy-defining corrections.
\begin{table}[t]
\centering
\caption{Principal configurations.  The tree ratio is relative to LGL--L3.  The printed value of $\beta_E$ is only a decimal label for the algebraic root in \eqref{eq:betaExactDefinition}.}
\label{tab:designs}
\footnotesize\setlength{\tabcolsep}{2.1pt}
\begin{tabular}{lcccccc}
\toprule
Method & $(c_1,c_2)$ & $\beta$ & $\rho(\Minf)$ & $|r_\infty|$ & $\Jstiff$ & tree ratio\\
\midrule
LGL--L3 & $(0.276393,0.723607)$ & $0.666667$ & $0.519109$ & $0$ & $0.519109$ & $1$\\
Accuracy-P40 & $(0.303155,0.721876)$ & $0.572261$ & $0.399570$ & $0.395123$ & $0.399570$ & $0.0979$\\
Certified-E7 & $(7/20,37/50)$ & algebraic & $0.472325$ & $0.177271$ & $0.472325$ & $0$\\
\bottomrule
\end{tabular}
\end{table}

\subsection{Accuracy-P40: work-oriented co-design}
Accuracy-P40 is a feasible best-found point of \eqref{eq:ParetoProblem} with the strict margin $\Jstiff=0.39957<0.40$.  Relative to LGL--L3,
\[
 \frac{\Jstiff(\mathrm{P40})}{\Jstiff(\mathrm{LGL-L3})}=0.770,
 \qquad
 \frac{\Jtree(\mathrm{P40})}{\Jtree(\mathrm{LGL-L3})}=0.0979.
\]
It is still generically sixth order, but the complete nonlinear principal defect is reduced by $90.2\%$ and the strong-stiff asymptotic sweep factor is smaller.  It is intended for calculations in which residual corrections and Krylov work dominate.  The local endpoint rule is not exactly $L$-stable, and the two-correction output is amplified as $z\to-\infty$; a third correction or a smaller macrostep is therefore used when relevant modes enter the far stiff tail.

\subsection{Certified-E7: order-oriented co-design}
Certified-E7 has the same three-subinterval, two-correction architecture but gains one classical order.  Its value of $\Jstiff$ lies between P40 and the LGL--L3 baseline.  It is therefore not the work-minimizing member; its purpose is to exploit the rank-one theorem so that one exact scalar cancellation enforces all 48 seventh-order conditions.  The method is defined by the integer polynomial and rational isolating interval, not by a rounded decimal parameter.  Its order guarantee also requires completion of both corrections and row residuals on the $\OO(\dt^8)$ local scale.

The two methods address different regimes.  P40 is recommended when repeated solution of the Hermite residual is the dominant cost.  Certified-E7 is recommended when truncation error dominates and the added order reduces the number of macrosteps.  Neither is proposed as a universal replacement for black-box stiff solvers, especially when $\Rtwo$ is expensive or unavailable.
\section{Numerical evidence}
\label{sec:numerics}
The experiments answer three questions tied directly to the claims: whether the algebraic method is genuinely seventh order on nonlinear systems, whether fixed-sweep damping behaves as predicted, and whether the strong-stiff objective correlates with reduced correction and Krylov work.

\subsection{Implementation and evaluation protocol}
All correction counts refer to sweeps performed after the H4 predictor.  Unless otherwise stated, the maximum is 25 corrections and the scaled residual \eqref{eq:residualStop} is used for stopping.  The software exposes a minimum correction count; Certified-E7 uses $K_{\min}=2$, and a third correction is used when far-stiff output damping is required.  Dense ODE tests solve the single-state equations by Newton iteration with analytic or finite-difference Jacobians.  Sparse phase-field tests use Newton--GMRES and ILU preconditioning.  The ILU factorization is reused within a stage, and the reported counts include Jacobian and preconditioner constructions.

Errors are measured against tighter reference integrations on the same spatial grid, so the tables isolate temporal and nonlinear-solve effects rather than spatial convergence.  Timing tests are repeated and reported as a median with the first and third quartiles.  The code records $N_f$, $N_g$, Newton iterations, GMRES iterations, and preconditioner builds, because wall time alone is strongly implementation dependent.

The two phase-field models are discretized periodically by centered finite differences.  They represent the nonconserved Allen--Cahn and conserved Cahn--Hilliard gradient-flow settings \cite{AllenCahn1979,CahnHilliard1958}; related SDC treatments of phase-field dynamics provide useful context \cite{YaoXiaXu2024}.  For Allen--Cahn we use a representative semilinear form
\[
 u_t=\varepsilon^2\Delta_hu+u-u^3,
\]
while the Cahn--Hilliard semidiscretization is written
\[
 u_t=\Delta_h\bigl(-\varepsilon^2\Delta_hu+u^3-u\bigr).
\]
The second derivative is evaluated through $\Rtwo(u)=\Rone_u(u)\Rone(u)$.  These finite-difference experiments are intended to test the time iteration; they are not proposed as state-of-the-art spatial phase-field solvers.

\subsection{General nonlinear seventh-order convergence}
The order claim is tested on two independent autonomous three-component systems using 60-decimal-digit arithmetic and a Newton tolerance of $10^{-48}$.  Both vector fields possess a known invariant solution curve but contain nonlinear transverse couplings, so branched elementary differentials are active away from that curve.  This avoids the weakness of a scalar or purely linear test, which can detect only chain-tree order.

For Test A, set $e_2=y-x^2$, $e_3=z-x^3$, and $\alpha=3/10$, and define
\begin{equation}
 f_A(x,y,z)=\begin{pmatrix}
 \alpha x+\frac15e_2+\sin(e_3)\\
 2\alpha y-\frac3{20}e_3+e_2^2\\
 3\alpha z+\frac1{10}e_2+xe_3
 \end{pmatrix},
 \qquad
 u_A(t)=\begin{pmatrix}e^{\alpha t}\\e^{2\alpha t}\\e^{3\alpha t}\end{pmatrix}.
 \label{eq:testA}
\end{equation}
For Test B, let $a=1/5$ and
\begin{equation}
 f_B(x,y,z)=\begin{pmatrix}
 ax^2+\frac17e_2+\frac19e_3^2\\
 2ax^3-\frac18e_3+e_2^2\\
 3ax^4+\frac1{11}e_2+xe_3
 \end{pmatrix},
 \qquad
 u_B(t)=\begin{pmatrix}q\\q^2\\q^3\end{pmatrix},\quad q=(1-at)^{-1}.
 \label{eq:testB}
\end{equation}
Direct substitution verifies both exact solutions.  Each experiment integrates to $T=1$ with exactly two corrections and step counts $1,2,4,8,16,32$.

\begin{figure}[t]
\centering
\begin{subfigure}{0.49\linewidth}
\centering
\includegraphics[width=\linewidth]{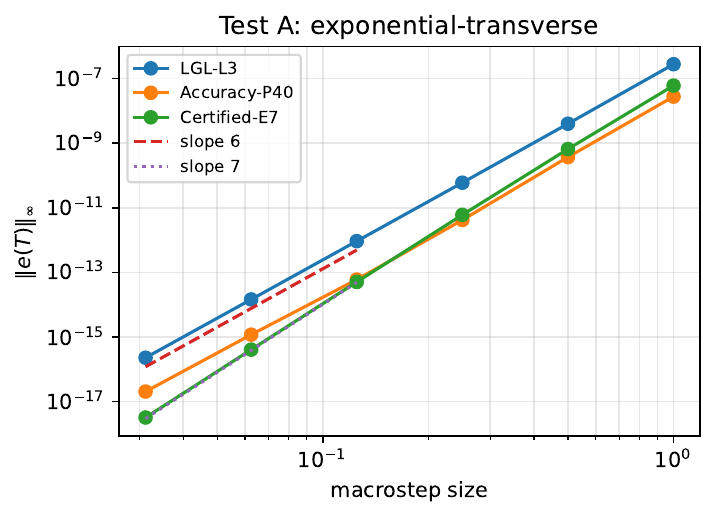}
\caption{Test A in \eqref{eq:testA}.}
\end{subfigure}\hfill
\begin{subfigure}{0.49\linewidth}
\centering
\includegraphics[width=\linewidth]{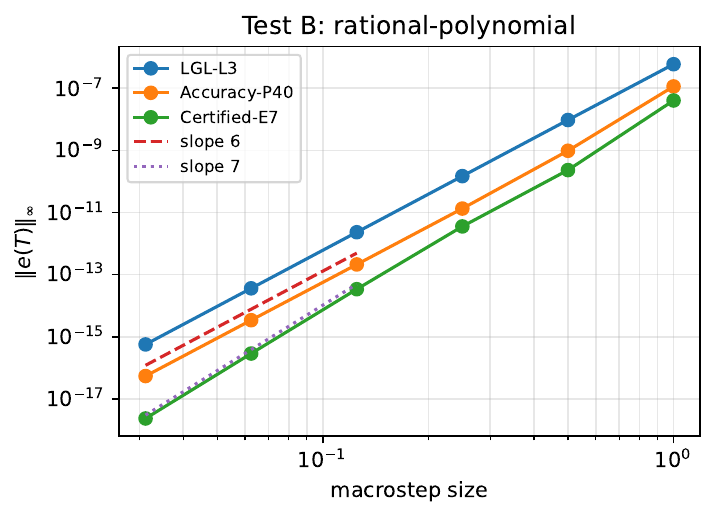}
\caption{Test B in \eqref{eq:testB}.}
\end{subfigure}
\caption{Complete six-level, 60-digit convergence data for both nonlinear tests.  Every plotted point corresponds to $1,2,4,8,16,$ or $32$ macrosteps; reference slopes six and seven are included.}
\label{fig:nonlinear}
\end{figure}

\begin{table}[H]
\centering
\caption{Independent nonlinear order checks at $T=1$.  The triples are the refinement rates on the last three intervals, corresponding to $\Delta t=1/8,1/16,1/32$; $e_{32}$ is the finest-step infinity-norm error.}
\label{tab:nonlinear}
\small
\resizebox{\linewidth}{!}{%
\begin{tabular}{lcccc}
\toprule
Method & Test A rates & $e_{32}^{A}$ & Test B rates & $e_{32}^{B}$\\
\midrule
LGL--L3 & $(6.011,5.995,5.994)$ & $2.28\times10^{-16}$ & $(6.000,5.999,6.000)$ & $5.64\times10^{-16}$\\
Accuracy-P40 & $(6.130,5.663,5.852)$ & $2.04\times10^{-17}$ & $(5.978,5.961,5.974)$ & $5.41\times10^{-17}$\\
Certified-E7 & $(6.901,6.959,6.982)$ & $3.21\times10^{-18}$ & $(6.711,6.881,6.945)$ & $2.33\times10^{-18}$\\
\bottomrule
\end{tabular}}
\end{table}

\Cref{fig:nonlinear,tab:nonlinear} show sixth-order convergence for LGL--L3, sixth-order convergence with a smaller principal error for Accuracy-P40, and convergence to order seven for Certified-E7 on both vector fields.  All three methods use the same nodes-per-step count, H4 predictor, and two completed corrections.  The additional order is therefore attributable to cancellation of the rank-one order-seven defect rather than to an added stage or sweep.  The exact 48-tree arithmetic in \cref{thm:rankone} is the coefficient-level proof; the two dynamical tests provide independent numerical cross-checks rather than replacing that proof.

\subsection{Fixed-sweep stability and an adaptive safeguard}
The distinction in \cref{thm:stoppedLimit} is visible before the PDE tests.  Although all reported new designs have $\rho(\Minf)<1$, the two-correction output limits in \cref{tab:stoppedStability} exceed one.  Thus the certified seventh-order statement and the strong-stiff iteration statement cannot be combined into a claim of two-sweep $A$-stability.  Accuracy-P40 has the largest sampled two-sweep negative-real interval among the principal methods, but its interval remains finite.

The software therefore separates an accuracy requirement from a damping safeguard.  Certified-E7 completes at least two corrections.  When a stiffness estimate places relevant modes beyond the sampled $K=2$ interval, the method either performs a third correction or reduces the macrostep.  This policy preserves the two-correction seventh-order construction while preventing the far-stiff output amplification from being hidden by a correction-error metric.

\subsection{One-dimensional Allen--Cahn mesh stiffness}
The mesh-scaling experiment keeps the time interval, macrostep count, and nonlinear tolerance fixed while the number of spatial unknowns increases.  The increasingly negative diffusive eigenvalues make the strong-stiff part of the spectrum more prominent.  \Cref{tab:mesh1d} show that Accuracy-P40 consistently uses fewer Newton and GMRES iterations and maintains a comparable or smaller error.

The iteration counts and the complete mesh-scaling curves are shown in \cref{tab:mesh1d,fig:mesh1d}.

\begin{table}[t]
\centering
\caption{One-dimensional Allen--Cahn mesh scaling.  Times are medians of repeated runs.}
\label{tab:mesh1d}
\small
\resizebox{\linewidth}{!}{%
\begin{tabular}{lrrrrrr}
\toprule
$n$ & Method & error & mean sweeps & Newton & GMRES & time (s)\\
\midrule
64 & LGL--L3 & $1.62\times10^{-7}$ & $6.5$ & 52 & 168 & 0.0690\\
   & Accuracy-P40 & $1.48\times10^{-7}$ & $4.5$ & 41 & 133 & 0.0555\\
128 & LGL--L3 & $1.65\times10^{-7}$ & $6.0$ & 49 & 199 & 0.0693\\
    & Accuracy-P40 & $1.50\times10^{-7}$ & $4.0$ & 38 & 150 & 0.0532\\
256 & LGL--L3 & $1.84\times10^{-7}$ & $5.0$ & 43 & 236 & 0.0704\\
    & Accuracy-P40 & $1.51\times10^{-7}$ & $3.5$ & 35 & 179 & 0.0564\\
\bottomrule
\end{tabular}}
\end{table}

\begin{figure}[t]
\centering
\includegraphics[width=0.70\linewidth]{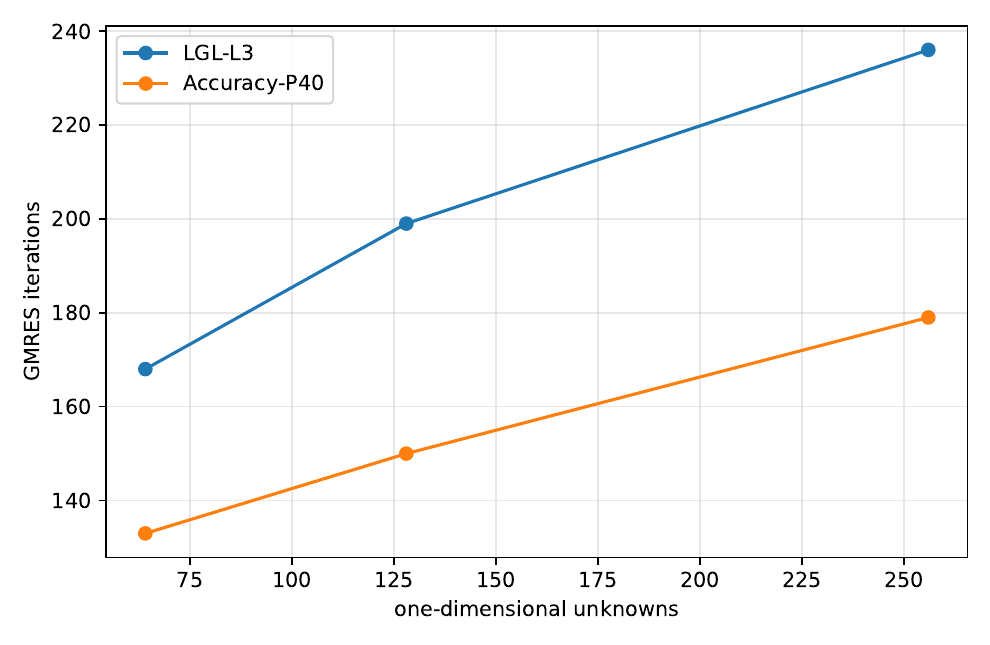}
\caption{One-dimensional Allen--Cahn work under mesh refinement.  The fixed time interval, macrostep count, and nonlinear tolerance isolate the effect of the widening diffusive spectrum.}
\label{fig:mesh1d}
\end{figure}

At $n=256$, the GMRES reduction is $24.2\%$ and the median time reduction in this implementation is about $19.8\%$.  The wall times are not monotone in $n$ because the problems are small and affected by sparse setup overhead; the iteration counts provide the more stable algorithmic comparison.

\subsection{Two-dimensional Cahn--Hilliard structure diagnostic}
The Cahn--Hilliard problem is more severe because the stiff linear part is fourth order.  The periodic semidiscretization also contains a conserved constant mode, so the experiment simultaneously tests a slowly evolving invariant mode and strongly damped fourth-order modes.  \Cref{tab:ch2d} reports accuracy, work, energy, and mass diagnostics.

\begin{table}[t]
\centering
\caption{Two-dimensional Cahn--Hilliard results on a $16^2$ periodic grid.  The energy increase column is the largest positive difference between recorded macrostep energies.}
\label{tab:ch2d}
\small
\resizebox{\linewidth}{!}{%
\begin{tabular}{lrrrrr}
\toprule
Method & error & mean sweeps & GMRES & max energy increase & max mass drift\\
\midrule
LGL--L3 & $4.29\times10^{-7}$ & $4.250$ & 535 & $0$ & $8.24\times10^{-17}$\\
Accuracy-P40 & $4.13\times10^{-7}$ & $2.750$ & 425 & $0$ & $7.89\times10^{-17}$\\
\bottomrule
\end{tabular}}
\end{table}

\begin{figure}[t]
\centering
\includegraphics[width=0.62\linewidth]{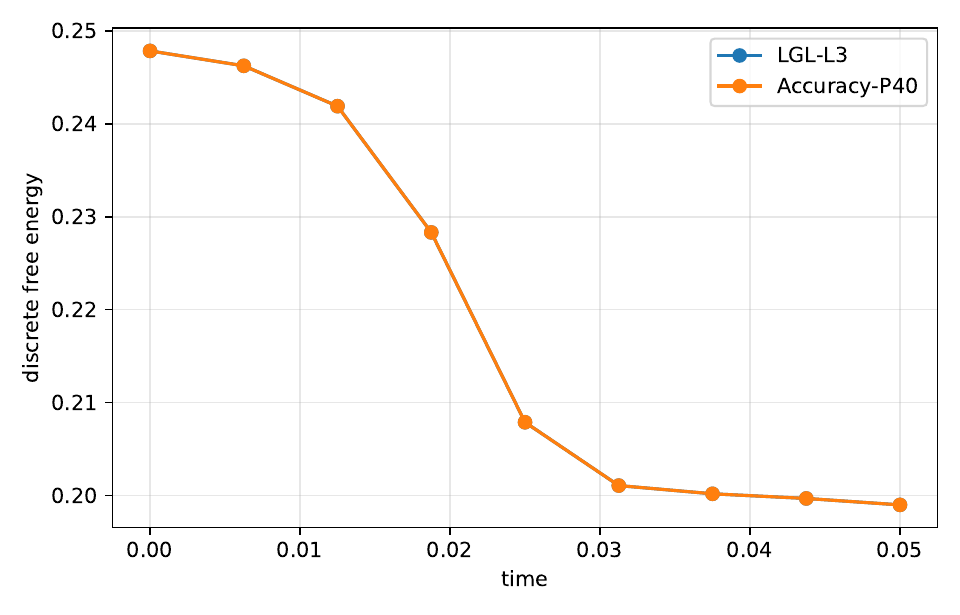}
\caption{Complete recorded discrete free-energy histories for the $16^2$ Cahn--Hilliard experiment.  Both histories decrease at every reported macrostep; this is an observed diagnostic, not an unconditional energy-stability theorem.}
\label{fig:chEnergy}
\end{figure}

Accuracy-P40 reduces the average correction count by $35.3\%$ and the GMRES count by $20.6\%$ relative to LGL--L3.  The numerical mean is preserved to roundoff and the recorded free energy decreases monotonically.  These properties arise from the spatial structure and accurate nonlinear solves; they are observations, not an unconditional energy-stability theorem for the time integrator.

\subsection{Interpretation and external boundary}
The nonlinear tests support the exact order claim, and the phase-field experiments support the residual-work claim.  They play different logical roles: the coefficient-level proof establishes all 48 order conditions, while the numerical tests show that the certified cancellation is visible on nonlinear dynamics with branched elementary differentials.  Likewise, the Allen--Cahn and Cahn--Hilliard tables do not prove mesh-uniform convergence or energy stability; they test whether the proposed strong-stiff objective correlates with practical correction and Krylov work.

The claim is architecture-specific: when two-derivative information is already available or moderately priced, node--sweep co-design can reduce Hermite correction work, and the rank-one defect structure permits an exact seventh-order member without adding a stage or correction.  No universal timing advantage over black-box stiff solvers is asserted.
\section{Discussion}
\label{sec:discussion}
\subsection{Innovation relative to prior work}
The novelty is not high-order Hermite deferred correction by itself.  The first innovation is the joint design of the node distribution and the sequential two-derivative endpoint preconditioner under a fixed, practically constrained architecture.  These variables enter the finite-sweep truncation error and the strong-stiff correction matrix together, so choosing classical nodes and then optimizing only the sweep solves a different problem.

The stronger innovation is the rank-one order-seven defect.  A generic seventh-order B-series has 48 rooted-tree conditions, and a scalar stability-series coefficient normally detects only the chain tree.  Here, a parameter-independent B-series argument proves that the H4 predictor defect is one-directional and that each leading correction acts by unary grafting with a common scalar stage recurrence.  The complete 48-dimensional defect therefore collapses to one dimension after two corrections.  This is what makes the scalar algebraic certificate mathematically legitimate.

The third innovation is the exactness of the certification.  Certified-E7 is not defined by an optimizer output rounded to many digits.  Rational nodes reduce the order-seven chain defect to a primitive quadratic polynomial; rational sign intervals isolate its roots, and a modular gcd calculation proves that the selected root cannot simultaneously cancel the order-eight chain defect.  The resulting method has order exactly seven, not merely observed order seven.

\subsection{Why the selected theory and experiments are indispensable}
Four theoretical components are essential to the central claim: the stopped-method lower bound identifies generic sixth order; the predictor-direction lemma and tree-independent propagation proposition establish rank one; the algebraic root certificate cancels the full seventh-order defect; and the order-eight coprimality argument proves exact rather than at-least seventh order.  The endpoint $A$-stability classification and stopped-output formula are indispensable to interpreting the co-design and avoiding an incorrect claim of two-sweep $A$-stability.

The indispensable numerical evidence is correspondingly compact.  Two independent high-precision nonlinear systems are needed because a scalar or linear problem cannot activate branched elementary differentials.  Fixed-sweep negative-axis data are needed to demonstrate the distinction between residual contraction and returned-output damping.  One diffusion-dominated mesh-scaling test and one fourth-order phase-field test are needed to connect the strong-stiff objective to correction and Krylov work.  The selected experiments are therefore evidence for the two stated advantages rather than a broad benchmark survey.

\subsection{Limitations}
The rank-one identity is specific to the H4 predictor, $s=3$, and two corrections; other predictors or correction counts can produce a higher-dimensional principal defect.  P40 is a best-found multistart design rather than a certified global minimizer.  The matrix-valued stiff theorem is asymptotic on an invertible stiff subspace and does not control every finite-parameter pseudospectrum or nonlinear Newton transient.  The phase-field experiments use moderate finite-difference models and do not establish unconditional energy stability, a maximum principle, or mesh-uniform stiff convergence.  Finally, exact order requires two completed corrections and sufficiently small row residuals; fixed nonlinear tolerances can destroy the asymptotic rate.

\section{Conclusions}
\label{sec:conclusions}
For the three-subinterval H4 predictor with two corrections, the complete order-seven nonlinear defect is rank one, so one scalar coefficient controls all 48 rooted-tree conditions.  Rational nodes and an isolated algebraic correction parameter cancel this coefficient, while coprimality with the order-eight chain polynomial proves that Certified-E7 has classical order exactly seven.  Accuracy-P40 uses the same architecture to reduce the complete principal-error norm and the strong-stiff correction factor, and the Allen--Cahn/Cahn--Hilliard tests show corresponding reductions in correction and Krylov work.  The stopped-output analysis clarifies that a third correction may still be needed for far-stiff damping.  These results make the two designs complementary: Certified-E7 is order-oriented, whereas Accuracy-P40 is work-oriented.

\appendix
\section{Exact coefficient audit for the rank-one certificate}
\label{app:audit}
This appendix gives a direct audit path for every algebraic quantity used in the proof.  Partition the Hermite matrices as $Q_1=[q_0\mid Q_{1,a}]$ and $Q_2=[\widehat q_0\mid Q_{2,a}]$.  On the scalar test equation the active dense stages satisfy $A_Q(z)U=b_Q(z)u^n$, while one correction has the affine form
\begin{equation}
 U^{[k+1]}=M_\beta(z)U^{[k]}+r_\beta(z)u^n,
 \qquad M_\beta(z)=I-P_\beta(z)^{-1}A_Q(z).
 \label{eq:auditIteration}
\end{equation}
Consequently,
\begin{equation}
 U^{[K]}=M_\beta(z)^KU^{[0]}+
 \sum_{\ell=0}^{K-1}M_\beta(z)^\ell r_\beta(z)u^n.
 \label{eq:auditStopped}
\end{equation}
For rational nodes, formal inversion of $P_\beta(z)$ and coefficient extraction from \eqref{eq:auditStopped} use only rational arithmetic.  At the nodes \eqref{eq:Enodes}, extraction of $[z^7]R_{3,2}-1/7!$ and $[z^8]R_{3,2}-1/8!$ gives exactly \eqref{eq:C7rational}--\eqref{eq:p8explicit}.  The rational witness \eqref{eq:genericWitness} follows from the same recurrence at $(c_1,c_2,\beta)=(1/4,3/4,2/3)$.

For the nonlinear audit, formulas \eqref{eq:fComposition}--\eqref{eq:exactTreeRecursion} are applied in increasing tree order.  At each stage, the H4 predictor and correction equation \eqref{eq:correction} are composed as formal B-series.  If the first defect is $h^r\lambda_m^{[k]}e_r$, coefficient collection at order $r+1$ gives exactly
\begin{equation}
 \lambda_m^{[k+1]}=\lambda_{m-1}^{[k+1]}
 +\sum_{j=0}^{s}q_{mj}^{(1)}\lambda_j^{[k]}
 -\Delta c_m\bigl((1-\beta)\lambda_{m-1}^{[k]}+\beta\lambda_m^{[k]}\bigr),
 \label{eq:auditStageRecursion}
\end{equation}
with $\lambda_0^{[k]}=0$.  The $h^2g$ terms first enter at order $r+2$, so \eqref{eq:auditStageRecursion} is tree independent and the only order-$(r+1)$ operation is unary grafting.  Starting with the nine order-five trees, two applications produce exactly nine supported order-seven trees and leave the other $48-9=39$ coefficients zero.  Substitution of the exact-flow coefficients then gives the multipliers in \eqref{eq:rankvalues}; direct summation gives $90$ without symmetry weights and $886$ with symmetry weights.

The accompanying code implements these recurrences with exact integers and rationals.  Its certificate file records the two primitive polynomials, the three rational sign evaluations in \eqref{eq:simpleIsolation}, the nonsquare discriminant, the Euclidean calculation \eqref{eq:modEuclid}, the nonzero resultant, all 48 order-seven checks, and the binary64 evaluation residual.  Thus the computer audit is reproducible, while the parameter-independent argument in \cref{lem:predictorDefect,prop:stagePropagation,thm:rankone} supplies the mathematical proof of the one-dimensional defect image.

\section*{Statements and Declarations}
\noindent\textbf{Funding.} This work was supported by the Henan Provincial Department of Education (Grant No.~26A110007), the Henan Provincial Department of Science and Technology (Grant No.~252300423500), and the Doctoral Research Fund of Henan Polytechnic University (Grant No.~B2024-60).

\printbibliography
\end{document}